\documentclass{amsart}

\usepackage{amsfonts,amssymb}
\usepackage{amsmath}
\usepackage{algorithm}

\newtheorem{thm}{Theorem}
\newtheorem{lem}[thm]{Lemma}
\newtheorem{remark}{Remark}
\newtheorem{cor}[thm]{Corollary}
\newtheorem{prop}[thm]{Proposition}

\newtheorem{example}{Example}

\begin{document}

\title[Absolute value equations]{Comments on finite termination of the generalized Newton method for absolute value equations}

\author{Chun-Hua Guo}
\address{Department of Mathematics and Statistics, 
University of Regina, Regina,
SK S4S 0A2, Canada} 
\email{chun-hua.guo@uregina.ca}

\subjclass{Primary 65H10; Secondary 90C33}

\keywords{Absolute value equation; Generalized Newton method; Global convergence; Finite termination.}

\begin{abstract}

We consider the generalized Newton method (GNM) for the absolute value equation (AVE)
$Ax-|x|=b$. The method has finite termination property whenever it is convergent, no matter whether the AVE has a unique solution. We prove that GNM is convergent whenever $\rho(|A^{-1}|)<1/3$. We also present new results for the case where 
$A-I$ is a nonsingular $M$-matrix or an irreducible singular $M$-matrix. When $A-I$ is an irreducible singular $M$-matrix, the AVE may have infinitely many solutions. In this case, we show that GNM always terminates with a uniquely identifiable solution, as long as the initial guess has at least one nonpositive component. 

\end{abstract}

\maketitle

\section{Introduction}
\label{sec1}

We consider the absolute value equation (AVE):
\begin{equation}\label{AVE}
Ax - |x|= b, 
\end{equation}
where $A\in {\mathbb R}^{n\times n}$ and $b\in {\mathbb R}^n$ are given, and $|\cdot |$ denotes absolute value.
The AVE can be obtained \cite{MM06,P09} by reformulating the linear complementarity
problem (LCP), which appears in many mathematical programming problems, and has thus received considerable attention in the optimization community. 
Conditions for the unique solvability of the AVE have been given in \cite{HM23,MM06,RHF14,WL18,ZW09} for example. 
Many numerical methods have been proposed for the AVE \eqref{AVE}; see, for example,  \cite{BFP16,CQZ,CYH23,GWL19,KM17,M07,M09,RHF14,ZH21,ZW09}. 
In particular, the generalized Newton method (GNM) proposed in \cite{M09} is structurally very simple and is easy to use. Moreover, it often terminates with an exact solution only after a few iterations. In this paper, we then give GNM some further considerations. 

We use  ${\rm diag} (v)$ to denote a diagonal matrix corresponding to a vector $v$. 
We use ${\rm sign} (x)$ to denote  a vector with components equal to
$1, 0$ or $-1$ depending on whether the corresponding component of $x$ is positive, zero or
negative. A generalized Jacobian $\partial |x|$ of $|x|$  is given by the diagonal
matrix $D(x)$:
\begin{equation}
D(x) = \partial |x| = {\rm diag} ({\rm sign} (x)). 
\end{equation}
The GNM for the AVE \eqref{AVE} is given in \cite{M09} by 
\begin{equation}\label{GNM}
x^{i+1} = (A- D(x^i ))^{-1}b, \quad i=0, 1, \ldots, 
\end{equation}
assuming that $A- D(x^i )$ is nonsingular. 

In this paper, we explain that GNM has the finite termination property whenever it is convergent, and prove a new convergence result for the GNM. 
We also consider a clearly described class of AVEs, which is shown to be the same as the class of AVEs recently studied in \cite{Chen23}. 
We improve some results in \cite{Chen23} and also obtain some new ones. 

{\bf Notation.} 
The $j$th component of a vector $u$ is denoted by $u_j$ and the $j$th diagonal entry of a diagonal matrix $D$ is denoted by $(D)_{j}$. The identity matrix is  denoted by $I$ or $I_n$ (to specify its dimension). 
We use $\|\cdot\|$   to denote the vector 2-norm and matrix 2-norm, and use  $\rho(A)$ to denote the spectral radius of a square matrix $A$. 
The superscript $T$
denotes the transpose of a vector or matrix, and $|A|$ denotes the matrix $\left[ |a_{ij}|\right] $ for any 
$n\times k$ matrix $A=\left[ a_{ij}\right]$.   
An $n\times k$  real matrix $A=\left[ a_{ij}\right]$ is called
nonnegative (positive) if $a_{ij}\geq 0$ $(a_{ij}>0)$ for all $i$ and $j$.
For real matrices $A$ and $B$ of the same size, we write $A\geq B$ ($A>B$) if $A-B$ is nonnegative
(positive).

\section{Preliminaries}
\label{sec2}

Using the connection between AVE and LCP, Zhang and Wei \cite{ZW09} show that the AVE \eqref{AVE} has a unique solution for every $b$  if the interval matrix $[A-I, A+I]$ is regular, i.e., all matrices $X$ with $A-I\le X\le A+I$ are nonsingular. It is shown in \cite[Theorem 3.2]{WL18} that this is actually a necessary and sufficient condition. Two commonly used sufficient conditions are $\|A^{-1}\|<1$ \cite{MM06} and 
$\rho(|A^{-1}|)<1$ \cite{RHF14}. 

A real square matrix $A$ is called a
$Z$-matrix if all its off-diagonal entries are nonpositive. Any $Z$-matrix $A$ can be written as $sI-B$ with $B\geq 0$; it is called a nonsingular $M$-matrix if $s>\rho (B)$, and a singular $M$-matrix if $s=\rho (B)$.  It is clear that $A$ is a nonsinguar (singular) $M$-matrix if and only if $A^T$ is so. 
A matrix $A$ is reducible if there is a permutation matrix $P$ such that 
$$
P^TAP=\left [\begin{array}{cc}
A_{11} & A_{12} \\
0 & A_{22}
\end{array} \right ], 
$$
where $A_{11}$ and $A_{22}$ are square matrices. 
A matrix $A$ is irreducible if it is not reducible. It is clear that $A$ is irreducible if and only if $A^T$ is irreducible. 
We rely heavily on results from the Perron--Frobenius theory of nonnegative matrices \cite{BPl94,V00}. 

\begin{lem}\cite[Theorems 2.7 and 2.20]{V00}\label{lem0}
Let $A\ge 0$ be a square matrix. Then $\rho(A)$ is an eigenvalue of $A$. 
If $A$ is also irreducible, then $\rho(A)$ is a positive eigenvalue of $A$ with a corresponding eigenvector $u>0$. 
The vector $u$ is uniquely determined up to a scalar multiple. 
\end{lem}

The next result follows directly from the previous one by using the definition of singular $M$-matrices.

\begin{lem}\cite[Theorem 6.4.16]{BPl94}\label{6416}
If $A$ is an irreducible singular $M$-matrix, then there is a vector  $u>0$ such that $Au=0$. 
The vector $u$ is uniquely determined up to a scalar multiple. 
\end{lem}

We also need the following results. 

\begin{lem}\cite[Theorem 6.2.3]{BPl94}
A $Z$-matrix $A$ is a nonsingular $M$-matrix if and only if $A^{-1}\ge 0$.
\end{lem}

\begin{lem}\cite[Theorem 6.2.7]{BPl94}\label{627}
An irreducible $Z$-matrix $A$ is a nonsingular $M$-matrix if and only if for
some $u>0$ the vector $Au$ is nonnegative and nonzero.
\end{lem}

\begin{lem}\cite[Theorem 2.21]{V00} \label{lem1}
For any $n\times n$ matrices $X$ and $Y$, $\rho(X)\le \rho(|X|)$. 
If $0\le X\le Y$, then $\rho(X)\le \rho(Y)$. 
\end{lem}

\begin{lem}\label{lem2} 
If $\rho(X)<1$ for $X\ge 0$ and $Y=(I-X)^{-1}X$, then 
$\rho(Y)=\rho(X)/(1-\rho(X))$. 
\end{lem}

\begin{proof}
Every eigenvalue $\mu$ of $Y$ is given by $\mu=\lambda/(1-\lambda)$, where $\lambda$ is an eigenvalue of $X$. 
Write $\lambda=ae^{i\theta}$ with $0\le a\le \rho(X)$. Then 
$$
|\mu|=\frac{a}{\sqrt{(1+a^2-2a\cos \theta}}\le \frac{a}{1-a}
\le \frac{\rho(X)}{1-\rho(X)}. 
$$
When $X\ge 0$, $\rho(X)$ is an eigenvalue of $X$ (see Lemma \ref{lem0}) and thus $\rho(X)/(1-\rho(X))$ is an eigenvalue of $Y$. Therefore, $\rho(Y)=\rho(X)/(1-\rho(X))$. 
\end{proof}

\section{Convergence and finite termination} 
\label{sec3}

We start with an improved statement about finite termination of GNM,  as compared to \cite[Proposition 4]{M09}.

\begin{prop}\label{ft1}
Suppose that $x^1, \ldots, x^i, x^{i+1}$ are defined by \eqref{GNM} for a given $x^0$. If 
$(D(x^{i+1}))_{j}=(D(x^{i}))_{j}$ for each $j$ with 
$(D(x^{i+1}))_{j}\ne 0$, then $x^{i+1}$ is a solution of the AVE \eqref{AVE}. 
\end{prop}

\begin{proof}
The proof is the same as that of \cite[Proposition 4]{M09}. That is, 
$$
0=(A - D(x^i ))x^{i+1} - b = Ax^{i+1} - D(x^{i+1})x^{i+1} - b= Ax^{i+1} - |x^{i+1}| - b.
$$
Note that  $D(x^i )x^{i+1} = D(x^{i+1})x^{i+1}$ holds without the stronger assumption $D(x^{i+1})=D(x^{i})$ in \cite[Proposition 4]{M09}.
\end{proof}

\begin{example}
Consider the AVE \eqref{AVE} with 
$$
A=\left [\begin{array}{cc} 
3 & -1\\
-1 & 3
\end{array} \right ], \quad b=\left [\begin{array}{c}
1 \\
-4
\end{array} \right ]. 
$$
With $x^0=[-1, -1]^T$, we find $x^1=[0, -1]^T$. So $x^1$ is a solution of \eqref{AVE} by Proposition \ref{ft1}, while \cite[Proposition 4]{M09} not yet applies. 
\end{example}

We will see later in this paper that Proposition \ref{ft1} can be used to improve a recent result in \cite{Chen23}. 
If the sequence $\{x^i\}$ is defined by \eqref{GNM} and converges to some vector 
$x^*$, then $x^*$ is clearly a solution of \eqref{AVE}. If $x^*$ has no zero components, then 
Proposition \ref{ft1} guarantees  
that $x^i=x^*$ for all $i$ sufficiently large. In other words, GNM has the finite termination property. 
If $x^*$ has  some  zero components, then Proposition \ref{ft1} by itself does not guarantee 
that $x^i=x^*$ for all $i$ sufficiently large. In fact, it does not rule out the situation where $x^*=[1, 0]^T$ and $x^i=[1, (-1)^i 2^{-i}]^T$.

The next result shows that the above situation would never happen.  It is a more general statement than \cite[Lemma 2]{ZW09}, where $[A-I, A+I]$ is regular. 

\begin{prop}\label{ft2}
Suppose that $x^1, \ldots, x^i, x^{i+1}$ are defined by \eqref{GNM} for a given $x^0$ and $x^*$ is a solution of \eqref{AVE}. 
If $(D(x^{i}))_j=(D(x^{*}))_j$ for each $j$ with 
$(D(x^{*}))_j\ne 0$, then $x^{i+1}=x^*$. 
\end{prop}

\begin{proof}
For later use, we present a simple proof as follows. 
\begin{eqnarray*}
x^{i+1}-x^*&=& (A-D(x^i))^{-1}b-x^*\\
&=& (A-D(x^i))^{-1}(b-(A-D(x^i))x^*)\\
&=& (A-D(x^i))^{-1}(D(x^i)x^*-|x^*|).
\end{eqnarray*}
If $(D(x^{i}))_j=(D(x^{*}))_j$ for each $j$ with 
$(D(x^{*}))_j\ne 0$, then $D(x^i)x^*-|x^*|=0$ and thus 
$x^{i+1}=x^*$.
\end{proof}

Although the sign pattern of $x^*$ is usually not known beforehand, Proposition \ref{ft2} guarantees finite termination of GNM whenever it is convergent, as 
already mentioned in \cite{ZW09} for the case where $[A-I, A+I]$ is regular. 

\begin{cor}
If the sequence $\{x^i\}$ is defined by \eqref{GNM} and converges to some vector $x^*$ then $x^*$ is a solution of \eqref{AVE} and $x^i=x^*$ for all $i$ sufficiently large. In other words, GNM has finite termination property whenever it is convergent, even if the AVE \eqref{AVE} has multiple solutions. 
\end{cor}

Sometimes, it is more convenient to show convergence and then know the finite termination. Sometimes, it is more convenient to show finite termination and then know the convergence. 
The following convergence result of GNM is known. 

\begin{thm}
Suppose that $\|A^{-1}\|<1/3$. Then for any $b$ and any starting point $x^0$ 
GNM \eqref{GNM} generates a sequence $\{x^i\}$  converging to 
the unique solution $x^*$ of the AVE \eqref{AVE}. Moreover, 
$$
\|x^{i+1}-x^*\|\le c  \|x^i-x^*\|
$$
for some constant $c<1$ and  all $i\ge 0$. 
\end{thm} 

\begin{remark}
This result is a special case of \cite[Theorem 2]{BFP16}. 
The same conclusion is reached in \cite[Proposition 7]{M09} under the assumption that $\|A^{-1}\|<1/4$ and $\|D(x^i)\|\ne 0$. It is clear from \cite{M09} that the condition $\|D(x^i)\|\ne 0$ is not needed. The condition $\|A^{-1}\|<1/4$ is imposed in \cite{M09} only because the inequality $\||x|-|y|\|\le 2 \|x-y\|$ is used in the proof while  $\||x|-|y|\|\le  \|x-y\|$
holds trivially. This has already been pointed out in \cite{LLW18}, where the assumption  $\|D(x^i)\|\ne 0$ is still kept. 
\end{remark}

We now prove a new convergence result.

\begin{thm}
Suppose that $\rho(|A^{-1}|)<1/3$. Then for any $b$ and any starting point $x^0$ 
GNM \eqref{GNM} generates a sequence $\{x^i\}$  converging to 
the unique solution $x^*$ of the AVE \eqref{AVE}. 
\end{thm} 

\begin{proof}
For each $i\ge 0$, $|A^{-1}D(x^i)|\le |A^{-1}|$. Therefore, by Lemma \ref{lem1}, 
$\rho(A^{-1}D(x^i))\le \rho(|A^{-1}D(x^i)|)\le \rho(|A^{-1}|)<1/3$, and 
then $(I-A^{-1}D(x^i))^{-1}$ exists. Moreover, 
$$
\left |(I-A^{-1}D(x^i))^{-1}\right |=\left |\sum_{k=0}^{\infty} (A^{-1}D(x^i))^k\right |\le \sum_{k=0}^{\infty} |A^{-1}|^k
=(I-|A^{-1}|)^{-1}. 
$$
Thus $A-D(x^i)=A(I-A^{-1}D(x^i))$ is nonsingular. Since $|XY|\le |X||Y|$ whenever 
$XY$ is defined, we have  
$$
|(A-D(x^i))^{-1}|\le |(I-A^{-1}D(x^i))^{-1}| |A^{-1}|\le (I-|A^{-1}|)^{-1}|A^{-1}|.  
$$
Let $B=(I-|A^{-1}|)^{-1}|A^{-1}|\ge 0$.

We continue with the expression for $x^{i+1}-x^*$ in the proof of Proposition \ref{ft2}. 
\begin{eqnarray*}
x^{i+1}-x^*&=&  (A-D(x^i))^{-1}(D(x^i)x^*-|x^*|)\\
&=& (A-D(x^i))^{-1}(D(x^i)(x^*-x^i+x^i)-|x^*|)\\
&=& (A-D(x^i))^{-1}(D(x^i)(x^*-x^i)+|x^i|-|x^*|).
\end{eqnarray*}
Taking absolute value on both sides, we have for each $i\ge 0$ 
\begin{eqnarray*}
|x^{i+1}-x^*| &\le &  |(A - D(x^i))^{-1}|(||x^i| - |x^*|| + |D(x^i)| |x^i- x^*|)\\
&\le &  |(A - D(x^i))^{-1}|(2|x^i - x^*|)\\
&\le &  2 B |x^i - x^*|. 
\end{eqnarray*}
Then for each $i\ge 0$ 
$$
|x^{i}-x^*| \le (2B)^{i}|x^0-x^*|. 
$$
So $\{x^i\}$ converges to $x^*$ as long as $\rho(2B)=2\rho(B)<1$. 
However, by Lemma \ref{lem2}, 
$$
\rho(B)=(1-\rho(|A^{-1}|))^{-1}\rho(|A^{-1}|)<\frac{\frac{1}{3}}{1-\frac{1}{3}}=\frac{1}{2}. 
$$
So $\rho(2B)<1$, as required. 
\end{proof}

Note that the conditions $\|A^{-1}\|< \frac{1}{3}$ and $\rho(|A^{-1}|)<1/3$ are quite different. 
The difference of $\|A^{-1}\|$ and $\rho(|A^{-1}|)$ has been illustrated by many examples \cite{ZW09,RHF14}. 
Here we add one more example, which shows that the difference can be huge even when the entries of $A$ are all of moderate size. 

\begin{example}
Let 
$$
A=\left [\begin{array}{cccc}
1 & -1 & \cdots &  -1\\
0 & 1 &\ddots & \vdots \\
\vdots & \ddots & \ddots & -1\\
0 & \cdots & 0 & 1
\end{array}
\right ]_{n\times n}. 
$$
Then 
$$
A^{-1}=\left [\begin{array}{ccccc}
1 & 1 & 2 & \cdots &  2^{n-2}\\
0 & 1 & 1 & \ddots & \vdots \\
\vdots & \ddots  &   1 & \ddots & 2\\
\vdots &  & \ddots  &   \ddots & 1\\
0 & \cdots & \cdots & 0 & 1
\end{array}
\right ]_{n\times n}. 
$$
So $\rho(|A^{-1}|)=1$ and $\|A^{-1}\|\ge \|A^{-1} e_n\|>2^{n-2}$,  where $e_n=[0, \ldots, 0, 1]^T$. 
\end{example}

\section{AVEs associated with $M$-matrices}
\label{sec4}

In \cite{Chen23}, the AVE \eqref{AVE} is studied under any one of the following two assumptions: 
\newline 
(A1): $A-I$ is a nonsingular $M$-matrix.  
\newline
 (A2): There exists a vector $v>0$ in ${\mathcal{N}}(A^T-I)$ and $A - I + D$ is a nonsingular $M$-matrix for all
diagonal matrices $D = {\rm diag}(d)$ such that $d \ge 0$ and $d \ne  0$. 

Two issues have been raised in \cite{Chen23}. 
The first issue is how to verify assumption (A2). This issue is resolved by the following result. 

\begin{prop}\label{prop1}
Assumption (A2)  holds if and only if $A-I$ is an irreducible singular $M$-matrix. 
\end{prop}

\begin{proof}
Suppose $A-I$ is an irreducible singular $M$-matrix. Then $A^T-I$ is also an irreducible singular $M$-matrix. Therefore, 
there are  $u, v>0$ such that $(A-I)u=0$ and $(A^T-I)v=0$ (see Lemma \ref{6416}). For all diagonal matrices $D ={\rm diag}(d)$ such that $d \ge 0$ and $d \ne  0$, we have $(A-I+D)u\ge 0$ and $(A-I+D)u\ne 0$. Therefore, $A-I+D$ is a nonsingular $M$-matrix (see Lemma \ref{627}).
So assumption (A2) holds. 

Now suppose assumption (A2) holds. By letting $d\to 0$, we know that $A-I$ must be an $M$-matrix (singular or nonsingular). Since $A-I$ is singular by assumption, $A-I$ is a singular $M$-matrix. 
If $A-I$ is reducible, then there is a permutation matrix $P$ such that 
$$
P^T(A-I)P=\left [\begin{array}{cc} 
M_{11} & M_{12} \\
0 & M_{22}
\end{array}
\right ], 
$$
where $M_{11}\in {\mathbb R}^{m\times m}$ ($1<m<n$) and $M_{22}$ is singular. 
Let 
$$
D=\left [\begin{array}{cc} 
I_m  & 0\\
0 &  0
\end{array}
\right ]. 
$$
Then 
$$
P^T(A-I +PDP^T)P=\left [\begin{array}{cc} 
M_{11}+I_m & M_{12} \\
0 & M_{22}
\end{array}
\right ]
$$
is still singular. Thus $A-I+PDP^T$ is singular, contradictory to assumption (A2). This shows that $A-I$ is irreducible. 
\end{proof}

When $A-I$ is an irreducible singular $M$-matrix, there is $v>0$ such that $(A^T-I)v=0$ and $v$ is uniquely determined up to a scalar multiple (see Lemma \ref{6416}). 
We now have a refined statement of \cite[Theorem 3.1]{Chen23}. 

\begin{thm}\label{thmn}
If $A-I$ is a nonsingular $M$-matrix, then GNM \eqref{GNM} converges to an exact solution of \eqref{AVE} in at most $n+2$ iterations. 
If $A-I$ is an irreducible singular $M$-matrix and $v^Tb<0$, then for any $x^0$ with $D(x^0)\ne I$, 
GNM \eqref{GNM} converges to an exact solution of \eqref{AVE} in at most $n+1$ iterations. 
\end{thm}

\begin{proof}
In view of Proposition \ref{prop1}, it is already proved in \cite{Chen23} 
that the sequence $\{x^i\}$ is well defined and 
$D(x^{k+1}) \ge  D(x^k)$  ($k = 1, 2, \ldots$). 
At the first sight, the maximal number of iterations to reach an exact solution could be $2n+2$. This would occur in the hypothetical case where  
all components of $D(x^1)$ are $-1$ and in each iteration only one component changes (and it is an increase by $1$). In that case all components of 
$D(x^{2n+1})$ are $1$ and then $x^{2n+2}$ is an exact solution. However, Proposition \ref{ft1} says that we have termination if one $-1$ is increased to $0$ and no other components increase at the same iteration. This means that GNM terminates in at most $n+2$ iterations. 
In the second case, $D(x^k)$ has at least one negative component for each $k\ge 1$ (see \cite{Chen23}). Therefore, GNM terminates in at most $n+1$ iterations. 
\end{proof}

\begin{remark} It is important to note that $D(x^1)\ge D(x^0)$ is not true in general. The situation in \cite{BC08} is similar. The algorithm there terminates in at most $n+2$ iterations for case 1 there and at most $n+1$ iterations for case 2 there
(it was stated in \cite{BC08} that the algorithm terminates in at most $n+1$ iterations for both cases). 
For the AVE \eqref{AVE}, when $A-I$ is a nonsingular $M$-matrix and $n=2$, the sign change (starting with $x^0$ and ending with $x^*$) could theoretically be 
$$
\left [\begin{array}{c}
+\\
-
\end{array} \right ]\longrightarrow \left [\begin{array}{c}
-\\
-
\end{array} \right ]\longrightarrow \left [\begin{array}{c}
+\\
-
\end{array} \right ]\longrightarrow \left [\begin{array}{c}
+\\
+
\end{array} \right ]\longrightarrow \left [\begin{array}{c}
+\\
+
\end{array} \right ]. 
$$
This may never happen, but cannot be ruled out by the proof technique initiated in \cite{BC08}. 
Without the special assumption on $A$, it is easy to find an example for which GNM needs $4$ iterations to find an exact solution when $n=2$. 
One such example is 
$$
A=\left [\begin{array}{cc}
0.59 & 1.02\\
0.15 & 0.67
\end{array} \right ], \quad b=\left [\begin{array}{c}
1.68\\
0.05
\end{array} \right ], \quad x^0=\left [\begin{array}{c}
-0.46\\
-0.61
\end{array} \right ] . 
$$
\end{remark} 

When $A-I$ is a nonsingular $M$-matrix, $[A-I, A+I]$ is regular and thus the AVE \eqref{AVE} has a unique solution. 
When $A-I$ is an irreducible singular $M$-matrix, $[A-I, A+I]$ is not regular and thus it is impossible for the AVE \eqref{AVE} to have a unique solution for all $b$. 
In view of Proposition \ref{prop1}, it is already proved in \cite{Chen23} that \eqref{AVE} has a unique solution when $A-I$ is an irreducible singular $M$-matrix and $v^Tb<0$, 
and that \eqref{AVE} has no solution when $A-I$ is an irreducible singular $M$-matrix and $v^Tb>0$. 
We can also prove the following result, which resolves the second issue raised in \cite{Chen23}.  

\begin{prop}
If $A-I$ is an irreducible singular $M$-matrix and $v^Tb=0$, 
then the AVE \eqref{AVE} has infinitely many solutions. 
\end{prop}

\begin{proof}
When $v^Tb = 0$,  $b$ is in the range of $A- I$. So $(A -I )w = b$ for some vector $w$. 
Since $A-I$ is an irreducible singular $M$-matrix, there is $u>0$ such that $(A-I)u=0$. 
Let $\alpha  \le  \min_i \frac{w_i}{u_i}$ and define 
$x(\alpha) = w-\alpha u$. Then $x(\alpha) \ge 0$ is a solution of \eqref{AVE} since 
$Ax(\alpha)-|x(\alpha)|=(A- I )x(\alpha) = b$. 
\end{proof}

It turns out that much more can be said when $A-I$ is an irreducible singular $M$-matrix with $v^Tb=0$. 

\begin{thm}
If $A-I$ is an irreducible singular $M$-matrix and $v^Tb=0$, 
then every solution of the AVE \eqref{AVE} is nonnegative and exactly one solution has one or more zero components. 
For any $x^0$ with $D(x^0)\ne I$,  GNM \eqref{GNM}  is well defined and converges to the unique solution with one or more zero components 
in at most $n+1$ iterations. 
\end{thm}

\begin{proof}
Let $x$ be any solution of \eqref{AVE}. Then $b=(A-D(x))x=(A-I+I-D(x))x$ and 
$(I-D(x))x=b-(A-I)x$. Thus $v^T(I-D(x))x=0$, which means that $x$ cannot have any negative components. 

For any $x^0$ with $D(x^0)\ne I$, $x^1$ is defined since $A-D(x^0)=A-I+I-D(x^0)$ is a nosingular $M$-matrix. We have 
$b=(A-D(x^0))x^1=(A-I+I-D(x_0))x^1$ and then $v^T(I-D(x^0))x^1=0$. So $x^1$ has at least one nonpositive component. 
Thus $D(x^1)\ne I$ and iteration \eqref{GNM} continues. 
The sequence $\{x^i\}$ is then well defined. Moreover, $D(x^{i+1})\ge D(x^i)$ for all $i\ge 1$ exactly as for the case $v^Tb<0$ in \cite{Chen23}. 
It follows that $x^i=\hat{x}$ is a solution of \eqref{AVE} with at least one zero component, for some $i\le n+1$, by Proposition \ref{ft1} and the explanation in the proof of Theorem \ref{thmn}. 

We now show that \eqref{AVE} has exactly one solution that has one or more zero components. 
Let $\bar{x}$ be any such solution. Applying GNM \eqref{GNM} with $x^0=\bar{x}$ gives $x^{n+1}=\bar{x}$. 
We now change the first component of $b$ from $b_1$ to $b_1-\epsilon$ (with $\epsilon>0$), and obtain a new vector $b(\epsilon)$. 
Since $v^Tb(\epsilon)<0$, the equation $Ax-|x|=b(\epsilon)$ has a unique solution $x(\epsilon)$. 
Applying GNM \eqref{GNM} with $x^0=\bar{x}$ to the equation $Ax-|x|=b(\epsilon)$
gives $x^{n+1}(\epsilon)=x(\epsilon)$ (see Theorem \ref{thmn}). Since $\lim_{\epsilon \to 0}x^{n+1}(\epsilon)=x^{n+1}=\bar{x}$, 
we have $\lim_{\epsilon \to 0}x(\epsilon)=\bar{x}$. Therefore,  $\bar{x}$ is uniquely determined. 
\end{proof}

\section{Conclusion}
\label{sec5}

In this paper, we have explained that GNM has finite termination property whenever it has convergence, regardless whether or not the AVE \eqref{AVE} has a unique solution. 
We have proved a new convergence result for GNM. 
We have also obtained some new results for the case where $A-I$ is a nonsingular $M$-matrix or an irreducible singular $M$-matrix. 
In particular, when $A-I$ is an irreducible  singular $M$-matrix and the AVE has more than one solutions, 
we have shown that the AVE then has infinitely many solutions, all of them are nonnegative. Moreover,  exactly one of them has some zero components, and this uniquely identified solution can be found by GNM in at most $n+1$ iterations, starting with any $x^0$ with at least one nonpositive component.

\section*{Acknowledgments} 
This work  was supported in part by an NSERC Discovery Grant RGPIN-2020-03973. \\

\end{document}